\documentclass[reqno]{amsart}
\usepackage{amsfonts}
\usepackage{amssymb}
\numberwithin{equation}{section}
\theoremstyle{plain}
 \newtheorem{thm}{Theorem}[section]
 \newtheorem{lem}[thm]{Lemma}
 
 \newtheorem{prop}[thm]{Proposition}
\theoremstyle{definition}
 
\theoremstyle{remark}
 
 \newtheorem{ex}[thm]{Example}

\begin{document}
\allowdisplaybreaks

\begin{flushleft}
{\Large {\bf Description of limits of ranges of iterations of 
stochastic integral mappings of infinitely divisible distributions}}

\bigskip
{\bf Ken-iti Sato}

\bigskip
{\small Hachiman-yama 1101-5-103, Tenpaku-ku, Nagoya, 468-0074 Japan\\
{\it E-mail address}:  ken-iti.sato@nifty.ne.jp\\
{\it URL}: http://ksato.jp/}
\end{flushleft}

\bigskip
\noindent {\bf Abstract}. For  infinitely divisible distributions $\rho$ on 
$\mathbb{R}^d$ the stochastic integral mapping $\Phi_f\,\rho$ is defined 
as the distribution of improper stochastic integral $\int_0^{\infty-}
f(s) dX_s^{(\rho)}$, where $f(s)$ is a non-random function and
$\{ X_s^{(\rho)}\}$ is a L\'evy process on $\mathbb{R}^d$ with
distribution $\rho$ at time $1$.  For three families of functions $f$
with parameters, the limits of the nested sequences of the ranges
of the iterations $\Phi_f^n$ are shown to be some subclasses, with
explicit description, of
the class $L_{\infty}$ of completely selfdecomposable distributions.
In the critical case of parameter $1$, the notion of weak mean $0$ plays
an important role.  Examples of $f$ with different limits of the ranges 
of $\Phi_f^n$ are also given.

\medskip
\noindent {\small {\it Short title}.  Limits of ranges of iterations of 
stochastic integral maps
}

\bigskip
\section{Introduction}

Let $ID=ID(\mathbb{R}^d)$ be the class of  infinitely divisible 
distributions on $\mathbb{R}^d$,
where $d$ is a fixed finite dimension.  For a real-valued locally 
square-integrable function $f(s)$ on $\mathbb{R}_+=[0,\infty)$, let
\[
\Phi_f\,\rho=\mathcal L\left(\int_0^{\infty-} f(s)dX_s^{(\rho)}\right),
\]
the law of the improper stochastic integral $\int_0^{\infty-} f(s)
dX_s^{(\rho)}$ 
with respect to the L\'evy process $\{ X_s^{(\rho)}\colon s\geq0\}$ 
on $\mathbb{R}^d$ with 
$\mathcal L(X_1^{(\rho)})=\rho$.  This integral is the limit in 
probability of $\int_0^t f(s)dX_s^{(\rho)}$ as $t\to\infty$.  The domain 
of $\Phi_f$, denoted by
$\mathfrak D(\Phi_f)$, is the class of $\rho\in ID$ such that this 
limit exists.
The range of $\Phi_f$ is denoted by $\mathfrak R(\Phi_f)$.
If $f(s)=0$ for $s\in (s_0,\infty)$, then $\Phi_f\,\rho=\mathcal L
\big (\int_0^{s_0} f(s)
dX_s^{(\rho)}\big )$ and  $\mathfrak D(\Phi_f)=ID$.  For many choices 
of $f$, the
description of $\mathfrak R(\Phi_f)$ is known; they are quite diverse.
A seminal example is $\mathfrak R(\Phi_f)=L=L(\mathbb{R}^d)$, the class 
of selfdecomposable
distributions on $\mathbb{R}^d$, for $f(s)=e^{-s}$ (Wolfe (1982), 
Sato (1999), Rocha-Arteaga
and Sato (2003)).
The iteration $\Phi_f^n$ is defined by $\Phi_f^1=\Phi_f$ and 
$\Phi_f^{n+1}\rho=\Phi_f(\Phi_f^n\rho)$
with $\mathfrak D(\Phi_f^{n+1})=\{\rho\in\mathfrak D(\Phi_f^n)\colon 
\Phi_f^n\,\rho\in\mathfrak D(\Phi_f)\}$. 
Then
\[
ID\supset \mathfrak R(\Phi_f)\supset \mathfrak R(\Phi_f^2)\supset \cdots .
\]
We define the limit class
\[
\mathfrak R_{\infty}(\Phi_f)=\bigcap_{n=1}^{\infty} \mathfrak R(\Phi_f^n).
\]
If $f(s)=e^{-s}$, then $\mathfrak R(\Phi_f^n)$ is the class of
$n$ times selfdecomposable distributions and $\mathfrak R_{\infty}(\Phi_f)$ is
the class $L_{\infty}$ of completely selfdecomposable distributions,
which is the smallest class that is closed under convolution and weak 
convergence and contains all stable distributions on $\mathbb{R}^d$.  
This sequence and the class $L_{\infty}$ were introduced by 
Urbanik (1973) and studied by Sato (1980) and others.
If $f(s)=(1-s) 1_{[0,1]}(s)$, then $\mathfrak R_{\infty}(\Phi_f)=L_{\infty}$,  
which was established
by Jurek (2004) and Maejima and Sato (2009); in this case $\mathfrak R(\Phi_f)$ 
is the class of
$s$-selfdecomposable distributions in the terminology of Jurek (1985). 
The paper of Maejima and Sato (2009) showed $\mathfrak R_{\infty}(\Phi_f)
=L_{\infty}$ in many cases
including $(1)$ $f(s)=(-\log s) 1_{[0,1]}(s)$,  $(2)$  $s=\int_{f(s)}^{\infty}
u^{-1} e^{-u} du$ $(0<s<\infty)$,  $(3)$  $s=\int_{f(s)}^{\infty} e^{-u^2} du$  
$(0<s<s_0=\sqrt{\pi}/2)$.
The classes $\mathfrak R(\Phi_f)$ corresponding to $(1)$--$(3)$ are the
Goldie--Steutel--Bondesson class $B$, the Thorin class $T$ (see 
Barndorff-Nielsen et al.\ (2006)),  
and the class $G$ of generalized type $G$ distributions, respectively.
These results pose a problem what classes other than $L_{\infty}$ can appear
as $\mathfrak R_{\infty}(\Phi_f)$ in general.  

For  $-\infty<\alpha<2$, $p>0$, and $q>0$, we consider the three families of
functions $\bar f_{p,\alpha}(s)$, $l_{q,\alpha}(s)$, and $f_{\alpha}(s)$ as in 
[S] (we refer to Sato (2010) as [S]).  
We define $\bar\Phi_{p,\alpha}$, $\Lambda_{q,\alpha}$, and $\Psi_{\alpha}$
to be the mappings $\Phi_f$ with $f(s)$ equal to these functions, respectively. 
In this paper we will prove the following theorem on the classes 
$\mathfrak R_{\infty}(\Phi_f)$ of those mappings.  The case $\alpha=1$ is
delicate.  There the notion of weak mean $0$ plays
an important role.

\begin{thm}\label{t1a}
{\rm(i)}  If $\alpha\leq0$, $p\geq1$, and $q>0$, then
\[
\mathfrak R_{\infty}(\bar\Phi_{p,\alpha})=\mathfrak R_{\infty}
(\Lambda_{q,\alpha})=\mathfrak R_{\infty}(\Psi_{\alpha})
=L_{\infty}.
\]
{\rm(ii)}  If\/ $0<\alpha<1$, $p\geq1$, and $q>0$, then
\[
\mathfrak R_{\infty}(\bar\Phi_{p,\alpha})=\mathfrak R_{\infty}
(\Lambda_{q,\alpha})=\mathfrak R_{\infty}(\Psi_{\alpha})=L_{\infty}^{(\alpha,2)}.
\]
{\rm(iii)}   If $\alpha=1$, $p\geq1$, and $q=1$, then
\[
\mathfrak R_{\infty}(\bar\Phi_{p,1})=\mathfrak R_{\infty}(\Lambda_{1,1})
=\mathfrak R_{\infty}(\Psi_1)=L_{\infty}^{(1,2)}
\cap\{\mu\in ID\colon \text{$\mu$ has weak mean $0$} \}.
\]
{\rm(iv)}   If\/ $1<\alpha<2$, $p\geq1$, and $q>0$, then
\[
\mathfrak R_{\infty}(\bar\Phi_{p,\alpha})=\mathfrak R_{\infty}
(\Lambda_{q,\alpha})=\mathfrak R_{\infty}(\Psi_{\alpha})
=L_{\infty}^{(\alpha,2)}\cap\{\mu\in ID\colon \text{$\mu$ has mean $0$}\}.
\]
\end{thm}

Let us explain the concepts used in the statement of Theorem \ref{t1a}.  
A distribution $\mu
\in ID$ belongs to $L_{\infty}$ if and only if its L\'evy measure $\nu_{\mu}$ is
represented as 
\[
\nu_{\mu}(B)=\int_{(0,2)} \Gamma_{\mu}(d\beta)\int_S \lambda_{\beta}^{\mu}
(d\xi) \int_0^{\infty} 1_B(r\xi) r^{-\beta-1} dr
\]
for Borel sets $B$ in $\mathbb{R}^d$, where $\Gamma_{\mu}$ is a measure on 
the open interval $(0,2)$ satisfying\linebreak
$\int_{(0,2)} (\beta^{-1}+(2-\beta)^{-1}) \Gamma_{\mu}(d\beta)<\infty$
and $\{\lambda_{\beta}^{\mu}\colon \beta\in(0,2)\}$ is a measurable family 
of probability measures on $S=\{\xi\in\mathbb{R}^d\colon |\xi|=1\}$.  This 
$\Gamma_{\mu}$ is uniquely determined by $\nu_{\mu}$ and  
$\{\lambda_{\beta}^{\mu}\}$ is determined by $\nu_{\mu}$
up to $\beta$ of $\Gamma_{\mu}$-measure $0$  (see [S] and Sato (1980)).  
For a Borel subset $E$ of the interval $(0,2)$, the class $L_{\infty}^E$ 
denotes, as in [S], the totality of $\mu\in L_{\infty}$ such that 
$\Gamma_{\mu}$ is concentrated on $E$.
The classes $L_{\infty}^{(\alpha,2)}$ and $L_{\infty}^{(1,2)}$ appearing in 
Theorem \ref{t1a} are for $E=(\alpha,2)$ and $(1,2)$, respectively.
Let $C_\mu (z)$ ($z\in\mathbb{R}^d$), $A_{\mu}$, and $\nu_{\mu}$ be the 
cumulant function, the Gaussian covariance matrix, and the L\'evy 
measure of $\mu\in ID$. 
A distribution $\mu\in ID$ is said to have weak mean $m_{\mu}$ if 
$\lim_{a\to\infty} \int_{1<|x|\leq a} x\nu_{\mu}(dx)$ exists in 
$\mathbb{R}^d$ and if
\[
C_{\mu}(z)=-\tfrac12 \langle z,A_{\mu}z\rangle +\lim_{a\to\infty} 
\int_{|x|\leq a} (e^{i\langle z,x\rangle}-1-i\langle z,x\rangle ) 
\nu_{\mu}(dx)+i\langle m_{\mu},z\rangle.
\]
This concept was introduced by [S] recently.
If  $\mu\in ID$ has mean $m_{\mu}$ (that is, $\int_{\mathbb{R}^d}|x|\mu(dx)
<\infty$ and $\int_{\mathbb{R}^d} x\mu(dx)=m_{\mu}$), then $\mu$ has weak 
mean $m_{\mu}$ (Remark 3.8 of [S]).

Section 2 begins with exact definitions of $f_{\alpha}$, $\bar f_{p,\alpha}$, 
and $l_{q,\alpha}$
and expounds existing results concerning $\mathfrak R_{\infty}(\Phi_f)$.
Then, in Section 3, we will prove Theorem \ref{t1a}.  
In Section 4 we will give examples of $\Phi_f$ for which 
$\mathfrak R_{\infty}(\Phi_f)$ is different from those appearing in 
Theorem 1.1.  Section 5 gives some concluding remarks.


\section{Known results}

Let $-\infty<\alpha<2$, $p>0$, and $q>0$ and let
\begin{align*}
\bar g_{p,\alpha}(t)&=\frac{1}{\Gamma(p)}
\int_t^1 (1-u)^{p-1} u^{-\alpha-1} du,\quad 0<t\leq 1,\\
j_{q,\alpha}(t)&=\frac{1}{\Gamma(q)} \int_t^1 (-\log u)^{q-1} u^{-\alpha-1} du,
\quad 0<t\leq 1,\\
g_{\alpha}(t)&=\int_t^{\infty} u^{-\alpha-1} e^{-u} du, \quad 0<t\leq\infty.
\end{align*}
Let $t=\bar f_{p,\alpha}(s)$ for $0\leq s<\bar g_{p,\alpha}(0+)$, 
$t=l_{q,\alpha}(s)$ for $0\leq s<j_{q,\alpha}(0+)$, and $t=f_{\alpha}(s)$ for 
$0\leq s<g_{\alpha}(0+)$ be 
the inverse functions of $s=\bar g_{p,\alpha}(t)$, $s=j_{q,\alpha}(t)$, and 
$s=g_{\alpha}(t)$, respectively.  They are continuous, strictly decreasing 
functions. If $\alpha<0$, then $\bar g_{p,\alpha}(0+)$, $j_{q,\alpha}(0+)$,
and $g_{\alpha}(0+)$ are
finite and we define $\bar f_{p,\alpha}(s)$, $l_{q,\alpha}(s)$, and 
$f_{\alpha}(s)$ to be zero for $s\geq \bar g_{p,\alpha}(0+)$, $j_{q,\alpha}(0+)$, 
and $g_{\alpha}(0+)$, respectively. 
Let  $\bar\Phi_{p,\alpha}$, $\Lambda_{q,\alpha}$, and $\Psi_{\alpha}$ denote 
$\Phi_f$ with 
$f=\bar f_{p,\alpha}$, $l_{q,\alpha}$, and $f_{\alpha}$, respectively. 
Let $K_{p,\alpha}$, $L_{q,\alpha}$, and $K_{\infty,\alpha}$ be the ranges of
$\bar\Phi_{p,\alpha}$, $\Lambda_{q,\alpha}$, and $\Psi_{\alpha}$, respectively.
These mappings and classes were systematically studied in Sato (2006) and [S].
In the following cases we have explicit expressions:
\begin{align*}
\bar f_{1,\alpha}(s)&=l_{1,\alpha}(s)= \begin{cases}
(1-|\alpha|s)^{1/|\alpha|}\,1_{[0,1/|\alpha|]}(s)\quad &\text{for }\alpha<0,\\
e^{-s}\quad &\text{for }\alpha=0,\\
(1+\alpha s)^{-1/\alpha}\quad &\text{for }\alpha>0, \end{cases} \\
\bar f_{p,-1}(s)&=\{ 1-(\Gamma(p+1)\,s)^{1/p}\} \,1_{[0,1/\Gamma(p+1)]}(s),\quad
p>0,\\
l_{q,0}(s)&=\exp (-(\Gamma(q+1)s)^{1/q}),\quad q>0,\\
f_{-1}(s)&=(-\log s)\,1_{[0,1]}(s).
\end{align*}
In the case $p=q=1$ we have $\bar\Phi_{1,\alpha}=\Lambda_{1,\alpha}$ and 
$K_{1,\alpha}=L_{1.\alpha}$, which are in essence treated earlier  
by Jurek (1988, 1989); $\bar\Phi_{1,\alpha}=\Lambda_{1,\alpha}$ were 
studied by Maejima et al.\ (2010), and Maejima and Ueda (2010b) with the notation 
$\Phi_{\alpha}$.  The mapping $\Lambda_{q,0}$ and the class $L_{q,0}$ with
$q=1,2,\ldots$ coincide with those introduced by Jurek (1983)
in a different form.   A variant of $\Psi_{\alpha}$ is found in 
Grigelionis (2007). 

A related family is 
\[
G_{\alpha,\beta}(t)=\int_t^{\infty} u^{-\alpha-1} e^{-u^{\beta}}du,\quad 
0<t\leq\infty,
\]
for $-\infty<\alpha<2$ and $\beta>0$. 
Let $t=G_{\alpha,\beta}^*(s)$ for $0\leq s<G_{\alpha,\beta}(0+)$ be the inverse
function of $s=G_{\alpha,\beta}(t)$.  If $\alpha<0$, then $G_{\alpha,\beta}(0+)$ is
finite and we define $G_{\alpha,\beta}^*(s)=0$ for $s\geq G_{\alpha,\beta}(0+)$.
Let $\Psi_{\alpha,\beta}$ denote $\Phi_f$ with $f=G_{\alpha,\beta}^*$.  This was 
introduced by Maejima and Nakahara (2009) and studied by Maejima and
Ueda (2010b) and, in the level of L\'evy measures, by Maejima et al.\ (2011b). 
Clearly, $\Psi_{\alpha,1}=\Psi_{\alpha}$. 
We have
\[
G_{-\beta,\beta}^* (s)=(-\log \beta s)^{1/\beta}\,1_{[0,1/\beta]}(s),\quad \beta>0.
\]
Earlier the mappings $\Psi_{0,2}$ and $\Psi_{-\beta,\beta}$ were treated in
Aoyama et al.\ (2008) and Aoyama et al.\ (2010), respectively; $\Psi_{-2,2}$ 
appeared also in Arizmendi et al.\ (2010). 

Maejima and Sato (2009) proved the following two results.

\begin{prop}\label{p2a}
Let $0<t_0\leq\infty$.  Let $h(u)$ be a positive decreasing function on 
$(0,t_0)$ such that $\int_0^{t_0}(1+u^2) h(u)du<\infty$.  Let 
$g(t)=\int_t^{t_0} h(u)du$ for $0<t\leq t_0$.  Let
$t=f(s)$, $0\leq s<g(0+)$, be the inverse function of $s=g(t)$ and let 
$f(s)=0$ for $s\geq g(0+)$.  Then $\mathfrak R_{\infty}(\Phi_f)=L_{\infty}$.
\end{prop}

\begin{prop}\label{p2b}
$\mathfrak R_{\infty}(\Psi_0)=L_{\infty}$.
\end{prop}

It follows from Proposition \ref{p2a} that $\mathfrak R_{\infty}(\Phi_f)
=L_{\infty}$ for $f=\bar f_{p,\alpha}$ with $p\geq1$
and $-1\leq\alpha<0$, $f=l_{q,\alpha}$ with $q\geq1$ and
$-1\leq\alpha<0$, $f=f_{\alpha}$ with $-1\leq\alpha<0$, 
and $f=G_{\alpha,\beta}^*$ with  $-1\leq\alpha<0$ and $\beta>0$.
The function $f_0$ for $\Psi_0=\Phi_{f_0}$ does not satisfy the condition in
Proposition \ref{p2a} but Proposition \ref{p2b} is proved using the identity 
$\Psi_0=\Lambda_{1,0}\Psi_{-1}=\Psi_{-1}\Lambda_{1,0}$.

In November 2007--January 2008, Sato wrote four memos, showing the part related to 
$\Psi_{\alpha}$ in (ii), (iii), and (iv) of Theorem \ref{t1a}. 
But assertion (iii) for $\Psi_1$ was shown with
the set $\{\mu\in ID\colon \text{$\mu$ has weak mean $0$} \}$ replaced by the set
of $\mu\in L_{\infty}$ satisfying some condition related to (4.6) of
Sato (2006). At that time the concept of weak mean was not yet introduced.
Those memos showed that some proper subclasses of $L_{\infty}$ appear
as limit classes $\mathfrak R_{\infty}(\Phi_f)$.

Sato's memos were referred to by a series of papers Maejima and Ueda (2009a, b,
2010a, b) and Ichifuji et al.\ (2010).
In Maejima and Ueda (2010a, c) they characterized 
$\mathfrak R(\Lambda_{1,\alpha}^n)$, $-\infty<\alpha<2$, 
for $n=1,2,\ldots$, in relation to a decomposability which they called
$\alpha$-selfdecomposability, and found
$\mathfrak R_{\infty}(\Lambda_{1,\alpha})$ for $-\infty<\alpha<2$.  
But the description of 
$\mathfrak R_{\infty}(\Lambda_{1,1})$ was similar to Sato's memos. 
In Maejima and Ueda (2010b) they showed that
$\Psi_{\alpha,\beta}$ with $-\infty<\alpha<2$ and $\beta>0$ satisfies
$\mathfrak R_{\infty}(\Psi_{\alpha,\beta})=\mathfrak R_{\infty}(\Psi_{\alpha})$,
under the condition that $\alpha\neq1+n\beta$ for $n=0,1,2,\ldots$.
For $\Psi_{0,2}$ and $\Psi_{-\beta,\beta}$ with $\beta>0$,
this result was earlier obtained by Aoyama et al.\ (2010).
Further it was shown in Maejima and Ueda (2009b) that 
$\mathfrak R_{\infty}(\Psi_{\alpha})=
\mathfrak R_{\infty}(\Lambda_{1,\alpha})$ for $-\infty<\alpha<2$.  An application of
the result in Maejima and Ueda (2010a) was given in Ichifuji et al.\ (2010).

If $f(s)=b\,1_{[0,a]}(s)$ for some $a>0$ and $b\neq 0$, then it is clear that
$\mathfrak R_{\infty}(\Phi_f)=\mathfrak R(\Phi_f)=ID$.  A first example of 
$\mathfrak R_{\infty}(\Phi_f)$
satisfying $L_{\infty}\subsetneqq \mathfrak R_{\infty}(\Phi_f) \subsetneqq ID$ 
was given by Maejima and Ueda (2009a); they showed that if $f(s)=b^{-[s]}$ for a given
$b>1$ with $[s]$ being the largest integer not exceeding $s$,  then
$\mathfrak R_{\infty}(\Phi_f)=L_{\infty}(b)$, the smallest class that is closed 
under convolution and weak convergence and contains all semi-stable
distributions on $\mathbb{R}^d$ with $b$ as a span; in this case 
$\mathfrak R(\Phi_f)$ is the class $L(b)$ of semi-selfdecomposable distributions 
on $\mathbb{R}^d$ with $b$ as a span.
See Sato (1999) for the definitions of semi-stability, semi-selfdecomposability,
and span.  See Maejima et al.\ (2000) for characterization of 
$L_{\infty}(b)$ as the limit of the class $L_n(b)$ of $n$ times 
$b$-semi-selfdecomposable distributions and for description of the 
L\'evy measures of distributions in $L_{\infty}(b)$.
Recall that $L_{\infty}\subsetneqq L_{\infty}(b)$.

The following result is deduced easily from [S].

\begin{prop}\label{p2c}
The assertions related to $\Lambda_{q,\alpha}$ in {\rm(i), (ii)}, and {\rm(iv)} of 
Theorem \ref{t1a} are true.
\end{prop}

Indeed, in [S], Theorem 7.3 says that $\Lambda_{q+q',\alpha}=\Lambda_{q',\alpha} 
\Lambda_{q,\alpha}$ 
for $\alpha\in(-\infty,1)\cup(1,2)$, $q>0$, and $q'>0$, and hence 
$\Lambda_{q,\alpha}^n=\Lambda_{nq,\alpha}$, and further, Theorem 7.11 combined with
Proposition 6.8 describes, for
$\alpha\in(-\infty,1)\cup(1,2)$,  the class $\bigcap_{q>0} L_{q,\alpha}$, which equals 
$\bigcap_{q=1,2\ldots} L_{q,\alpha}$.

\section{Proof of Theorem 1.1}

We prepare some lemmas. We use the terminology
in [S] such as radial decomposition, monotonicity of order $p$, and complete
monotonicity.  In particular, our complete monotonicity implies vanishing at infinity. 
The location parameter $\gamma_{\mu}$ of $\mu\in ID$ is defined by
\[
C_{\mu}(z)= -\tfrac12 \langle z,A_{\mu} z\rangle +\int_{\mathbb{R}^d}
(e^{i\langle z,x\rangle}-1-i\langle z,x\rangle
1_{\{|x|\leq 1\}}(x) ) \nu_{\mu}(dx)+i\langle \gamma_{\mu},z\rangle.
\]
Let $K_{p,\alpha}^{\mathrm e}$ $[$resp.\ $K_{\infty,\alpha}^{\mathrm e}]$ 
denote the class of distributions $\mu\in ID$ for which there exist $\rho\in ID$ and
a function $q_t$ from $[0,\infty)$ into $\mathbb{R}^d$ such that 
$\int_0^t f_{p,\alpha}(s)
dX_s^{(\rho)} -q_t$  [resp.\ $\int_0^t f_{\alpha}(s) dX_s^{(\rho)} -q_t$] 
converges in probability as $t\to\infty$ and the limit has distribution $\mu$.

\begin{lem}\label{l3a}
Let $-\infty<\alpha<2$ and $p>0$.
The domains of $\bar\Phi_{p,\alpha}$ and\/ $\Psi_{\alpha}$ are as follows:
\begin{align*}
&\mathfrak D(\bar\Phi_{p,\alpha})=\mathfrak D(\Psi_{\alpha})\\
&\quad=\begin{cases} 
ID \quad &\text{for }\alpha<0,\\
\{\rho\in ID\colon \int_{|x|>1} \log|x|\,\nu_{\rho}(dx)<\infty\} &\text{for }
\alpha=0,\\
\{\rho\in ID\colon \int_{|x|>1} |x|^{\alpha}\,\nu_{\rho}(dx)<\infty\} 
&\text{for }0<\alpha<1,\\
\{\rho\in ID\colon \int_{|x|>1} |x|\,\nu_{\rho}(dx)<\infty,\;\int_{\mathbb{R}^d}
x\,\rho(dx)=0, 
&{ }\\
\phantom{\rho\in IDID} \lim_{a\to\infty}\int_1^a s^{-1}ds \int_{|x|>s}
x\,\nu_{\rho}(dx)\text{ exists in }
\mathbb{R}^d \} &\text{for }\alpha=1,\\
\{\rho\in ID\colon \int_{|x|>1} |x|^{\alpha}\,\nu_{\rho}(dx)<\infty,
\;\int_{\mathbb{R}^d}x\,\rho(dx)=0 \} 
&\text{for }1<\alpha<2.
\end{cases}
\end{align*}
\end{lem}

This is found in Sato (2006) or Theorems 4.2, 4.4 and Propositions 4.6, 5.1 of [S].

\begin{lem}\label{l3b}
Let $-\infty<\alpha<2$ and $p>0$.
The class $K_{p,\alpha}^{\mathrm e}$ $[$resp.\ $K_{\infty,\alpha}^{\mathrm e}]$ 
is the totality of $\mu\in ID$ for which $\nu_{\mu}$ has a radial decomposition
$(\lambda_{\mu}(d\xi),\allowbreak u^{-\alpha-1}\,k_{\xi}^{\mu}(u)du)$ such that 
$k_{\xi}^{\mu}(u)$ 
is measurable in $(\xi,u)$ and, for $\lambda_{\mu}$-a.\,e.\ $\xi$, monotone of
order $p$ $[$resp.\ completely monotone$]$ on $\mathbb{R}_{+}^{\circ}=(0,\infty)$
in $u$.  The classes $K_{p,\alpha}$ and $K_{\infty,\alpha}$, that is, the ranges of 
$\bar\Phi_{p,\alpha}$ and\/ $\Psi_{\alpha}$, are as follows: 
\begin{align*}
K_{p,\alpha}&=\begin{cases} 
K_{p,\alpha}^{\mathrm e} \quad &\text{for }-\infty<\alpha<1,\\
\{\mu\in K_{p,1}^{\mathrm e} \colon \text{$\mu$ has weak mean $0$} \} 
&\text{for }\alpha=1,\\
\{\mu\in K_{p,\alpha}^{\mathrm e} \colon \text{$\mu$ has mean $0$} \} 
&\text{for }1<\alpha<2,
\end{cases} \\
K_{\infty,\alpha}&=\begin{cases} 
K_{\infty,\alpha}^{\mathrm e} \quad &\text{for }-\infty<\alpha<1,\\
\{\mu\in K_{\infty,1}^{\mathrm e} \colon \text{$\mu$ has weak mean $0$} \} &\text{for }\alpha=1,\\
\{\mu\in K_{\infty,\alpha}^{\mathrm e} \colon \text{$\mu$ has mean $0$} \} &\text{for }1<\alpha<2.
\end{cases}
\end{align*}
\end{lem}

See Theorems 4.18, 5.8, and 5.10 of [S]. 
Note that if $\mu$ is in $K_{\infty,\alpha}^{\mathrm e}$
or $K_{p,\alpha}^{\mathrm e}$ with $0<\alpha<2$, then $\int_{\mathbb{R}^d} 
|x|^{\beta}\mu(dx)<\infty$ for 
$\beta\in(0,\alpha)$ (Propositions 4.16 and 5.13 of [S]). 
It follows from the lemma above that  $K_{p,\alpha}^{\mathrm e}\supset 
K_{p',\alpha}^{\mathrm e}$ 
and $K_{p,\alpha}\supset K_{p',\alpha}$
for $p<p'$ and that $K_{\infty,\alpha}^{\mathrm e}=
\bigcap_{p>0} K_{p,\alpha}^{\mathrm e}$ and
$K_{\infty,\alpha}=\bigcap_{p>0} K_{p,\alpha}$.   
The notation of $K_{\infty,\alpha}^{\mathrm e}$ and $K_{\infty,\alpha}$ comes 
from this property.

\begin{lem}\label{l3c}
Let $\rho\in L_{\infty}$.\\
{\rm(i)}  Let $0<\alpha<2$.  Then $\int_{\mathbb{R}^d} |x|^{\alpha} \rho(dx)
<\infty$ if and only 
if $\Gamma_{\rho}((0,\alpha])=0$ and $\int_{(\alpha,2)} (\beta-\alpha)^{-1}
\,\Gamma_{\rho}(d\beta)<\infty$.\\
{\rm(ii)}  $\int_{|x|>1} \log |x|\,\rho(dx)<\infty$ if and only if 
$\int_{(0,2)}\beta^{-2}\,\Gamma_{\rho}(d\beta)<\infty$.
\end{lem}

\begin{proof}  Assertion (i) is shown in Proposition 7.15 of [S]. Since
\begin{align*}
&\int_{|x|>1} \log |x|\,\nu_{\rho}(dx)=\int_{(0,2)} \Gamma_{\rho}(d\beta)
\int_S \lambda_{\beta}^{\rho}(d\xi) \int_1^{\infty} (\log |r\xi|)r^{-\beta-1}dr\\
&\qquad=\int_{(0,2)} \Gamma_{\rho}(d\beta) \int_1^{\infty} (\log r)r^{-\beta-1}dr
=\int_{(0,2)} \beta^{-2} \Gamma_{\rho}(d\beta),
\end{align*}
assertion (ii) follows.
\end{proof}

\begin{lem}\label{l3d}
Let $\mu$ and $\rho$ be in $L_{\infty}^{(1,2)}$.  Suppose that 
$\Gamma_{\rho}(d\beta)=(\beta-1) b(\beta)\Gamma_{\mu}(d\beta)$ and
$\lambda_{\beta}^{\rho}=\lambda_{\beta}^{\mu}$ with a nonnegative measurable function
$b(\beta)$ such that $(\beta-1)^{-1} (b(\beta)-1)$ is bounded on $(1,2)$.
Then, $\int_1^a s^{-1}ds\int_{|x|>s}x\nu_{\rho}(dx)$ is convergent in $\mathbb{R}^d$ 
as $a\to\infty$ if and only if $\mu$ has weak mean $m_{\mu}$ for some $m_{\mu}$.
\end{lem}

\begin{proof}
Notice that $b(\beta)$ is bounded on $(1,2)$ and that
$\int_{|x|>1}|x|\nu_{\rho}(dx)<\infty$ by Lemma \ref{l3c}.
We have
\begin{align*}
&\int_1^a s^{-1}ds\int_{|x|>s}x\nu_{\rho}(dx)=\int_1^a s^{-1}ds
\int_{(1,2)}\Gamma_{\rho}(d\beta)\int_S \xi\lambda_{\beta}^{\rho}(d\xi)\int_s^{\infty}
r^{-\beta}dr\\
&\qquad=\int_{(1,2)}b(\beta) \Gamma_{\mu}(d\beta)\int_S 
\xi\lambda_{\beta}^{\mu}(d\xi)\int_1^a s^{-\beta}ds=I_1\quad\text{(say)}\\
\intertext{and}
&\int_{1<|x|\leq a} x\nu_{\mu}(dx)=\int_{(1,2)}\Gamma_{\mu}(d\beta)\int_S 
\xi\lambda_{\beta}^{\mu}(d\xi)\int_1^a r^{-\beta}dr=I_2\quad\text{(say)}.
\end{align*}
Hence
\[
I_1-I_2=\int_{(1,2)}(b(\beta)-1)\Gamma_{\mu}(d\beta)\int_S 
\xi\lambda_{\beta}^{\mu}(d\xi)\int_1^a r^{-\beta}dr.
\]
Since
\[
\left|(b(\beta)-1)\int_1^a r^{-\beta}dr\right| \leq(\beta-1)^{-1} |b(\beta)-1|
\]
and $\int_1^a r^{-\beta}dr$ tends to $(\beta-1)^{-1}$, $I_1-I_2$ is 
convergent in $\mathbb{R}^d$ as $a\to\infty$. Hence $I_1$ is convergent 
if and only if $I_2$ is convergent.
\end{proof}

\begin{lem}\label{l3e}
Let $f$ and $h$ be locally square-integrable functions on $\mathbb{R}_+$.
Assume that there is $s_0\in(0,\infty)$ such that $h(s)=0$ for $s\geq s_0$
and that $\Phi_h$ is one-to-one.  Then $\Phi_f\Phi_h=\Phi_h\Phi_f$.
\end{lem}

\begin{proof}
Let $f_t(s)=f(s)\,1_{[0,t]}(s)$.  Then $\Phi_{f_t} \Phi_h=\Phi_h
\Phi_{f_t}$ by Lemma 3.6 of Maejima and Sato (2009).  Let 
$\rho\in \mathfrak D(\Phi_f)$.
Then $\Phi_{f_t}\rho\to\Phi_f \rho$ as $t\to\infty$ by the definition of $\Phi_f$.
Hence $\Phi_h \Phi_{f_t}\rho\to\Phi_h \Phi_f \rho$ by (3.1) of Maejima and 
Sato (2009).
It follows that $\Phi_{f_t}\Phi_h\rho\to \Phi_h \Phi_f \rho$.  Since the 
convergence of $\int_0^t f(s)dX_s^{(\Phi_h \rho)}$ in law implies its convergence
in probability, $\Phi_h \rho$ is in $\mathfrak D(\Phi_f)$ and $\Phi_f\Phi_h\rho
=\Phi_h\Phi_f\rho$.  Conversely, suppose that $\rho\in ID$ satisfies $\Phi_h\rho
\in \mathfrak D(\Phi_f)$.  Then $\Phi_h\Phi_{f_t}\rho=\Phi_{f_t}\Phi_h\rho\to 
\Phi_{f}\Phi_h\rho$ as $t\to\infty$.  Looking at (3.8) of  Maejima and Sato (2009), 
we see that $\int_0^{s_0} h(s)\neq0$ from the one-to-one property of $\Phi_h$.
Hence $\{ \Phi_{f_t}\rho\colon t>0\}$ is precompact by the argument in pp.\,138--139
of  Maejima and Sato (2009).  Hence, again from the one-to-one property of $\Phi_h$, 
$\Phi_{f_t}\rho$ is convergent as $t\to\infty$, that is, $\rho\in\mathfrak D(\Phi_f)$.
\end{proof}

\begin{lem}\label{l3f}
Let $f$ be locally square-integrable on $\mathbb{R}_+$.  Suppose that there is 
$\beta\geq0$ such that any $\mu\in\mathfrak R(\Phi_f)$ has L\'evy measure $\nu_{\mu}$
with a radial decomposition $(\lambda_{\mu}(d\xi),\allowbreak 
u^{\beta}l_{\xi}^{\mu}(u)du)$
where $l_{\xi}^{\mu}(u)$ is measurable in $(\xi,u)$ and decreasing on
$\mathbb{R}_{+}^{\circ}$ in $u$.  Then
\[
\mathfrak R_{\infty}(\Phi_f)\subset \mathfrak R_{\infty}(\Lambda_{1,-\beta-1})
=L_{\infty}.
\]
\end{lem}

\begin{proof}
Clearly $l_{\xi}^{\mu}\geq0$ for $\lambda_{\mu}$-a.\,e.\ $\xi$.
Since $\int_{|x|>1} \nu_{\mu}(dx)<\infty$,  we have $\lim_{u\to\infty}
l_{\xi}^{\mu}(u)=0$ for $\lambda_{\mu}$-a.\,e.\ $\xi$.
Hence we can modify $l_{\xi}^{\mu}(u)$ in such a way that $l_{\xi}^{\mu}(u)$
is monotone of order $1$ in $u\in\mathbb{R}_{+}^{\circ}$.  Recall that a function is
monotone of order $1$ on $\mathbb{R}_{+}^{\circ}$ if and only if it is
decreasing, right-continuous, and vanishing at infinity (Proposition 2.11 
of [S]).  Then it follows from Theorem 4.18 or 6.12 of [S]
that 
\begin{equation}\label{3-1}
\mathfrak R(\Phi_f)\subset \mathfrak R(\Lambda_{1,-\beta-1}).
\end{equation}
Let us write $\Lambda=\Lambda_{1,-\beta-1}$ for simplicity.  
We have $\Phi_f\Lambda=\Lambda\Phi_f$ by virtue of Lemma \ref{l3e}, since $\Lambda$
is one-to-one (Theorem 6.14 of [S]).  If $\Phi_f\Lambda^n=\Lambda^n \Phi_f$ 
for some integer $n\geq1$, then
\[
\Phi_f\Lambda^{n+1}=\Phi_f\Lambda\Lambda^n=\Lambda\Phi_f\Lambda^n=\Lambda\Lambda^n 
\Phi_f=\Lambda^{n+1}\Phi_f.
\]
Hence $\Phi_f \Lambda^n=\Lambda^n \Phi_f$ for $n=1,2,\ldots$.  Now we claim that
\begin{equation}\label{3-2}
\mathfrak R(\Phi_f^n)\subset \mathfrak R(\Lambda^n)
\end{equation}
for $n=1,2,\ldots$.  Indeed, this is true for $n=1$ by \eqref{3-1}; if 
\eqref{3-2} is true for $n$, then any $\mu\in \mathfrak R(\Phi_f^{n+1})$ has
expression 
\[
\mu=\Phi_f^{n+1}\rho=\Phi_f\Phi_f^n \rho=\Phi_f\Lambda^n \rho'
=\Lambda^n\Phi_f \rho'=\Lambda^n \Lambda \rho''=\Lambda^{n+1} \rho''
\]
for some $\rho\in\mathfrak D(\Phi_f^{n+1})$, $\rho'\in\mathfrak D(\Lambda^n)$
with $\Phi_f^n\rho=\Lambda^n\rho'$, and $\rho''\in \mathfrak D(\Lambda)$ with 
$\Phi_f\rho'=\Lambda \rho''$, which means \eqref{3-2} for $n+1$.
It follows from \eqref{3-2} that $\mathfrak R_{\infty}(\Phi_f)\subset
\mathfrak R_{\infty}(\Lambda)$.
The equality $\mathfrak R_{\infty}(\Lambda)=L_{\infty}$ is from Proposition \ref{p2c}.
\end{proof}

\noindent {\it Proof of the part related to $\mathfrak R_{\infty}(\Psi_{\alpha})$ in
Theorem \ref{t1a}}.  
The result for $-1\leq\alpha\leq0$
is already known (see Propositions \ref{p2a} and \ref{p2b}).  But the proof
below also includes this case.
First, using Lemma \ref{l3b},
notice that Lemma \ref{l3f} is applicable to $\Phi_f=\Psi_{\alpha}$ 
and $\beta=(-\alpha-1)\lor 0$.

{\it Case 1} ($-\infty<\alpha<0$).  We have $\mathfrak D(\Psi_{\alpha})=ID$ in 
Lemma \ref{l3a}. Let us show that
\begin{equation}\label{A1-1}
\Psi_{\alpha}(L_{\infty})=L_{\infty}.
\end{equation}
Let $\rho\in L_{\infty}$ and $\mu=\Psi_{\alpha}\rho$.  Then for 
$B\in\mathcal B(\mathbb{R}^d)$,
where $\mathcal B(\mathbb{R}^d)$ is the class of Borel sets in $\mathbb{R}^d$,
\begin{align*}
\nu_{\mu}(B)&=\int_0^{\infty} ds\int_{\mathbb{R}^d} 1_B(f_{\alpha}(s)x)\nu_{\rho}(dx)
=\int_0^{\infty} t^{-\alpha-1}e^{-t}dt\int_{\mathbb{R}^d} 1_B(tx)\nu_{\rho}(dx)\\
&=\int_0^{\infty} t^{-\alpha-1}e^{-t}dt\int_{(0,2)} \Gamma_{\rho}(d\beta)
\int_S \lambda_{\beta}^{\rho}(d\xi)\int_0^{\infty} 1_B(tr\xi)r^{-\beta-1} dr\\
&=\int_{(0,2)} \Gamma(\beta-\alpha) \Gamma_{\rho}(d\beta) 
\int_S \lambda_{\beta}^{\rho}(d\xi)\int_0^{\infty} 1_B(u\xi)u^{-\beta-1} du.
\end{align*}
Hence $\mu\in L_{\infty}$ with
\begin{equation}\label{A1-2}
\Gamma_{\mu}(d\beta)=\Gamma (\beta-\alpha) \Gamma_{\rho}(d\beta)\quad\text{and}\quad
\lambda_{\beta}^{\mu}=\lambda_{\beta}^{\rho}.
\end{equation}
Let us show the converse.  
Let $\mu\in L_{\infty}$.  In order to find $\rho\in L_{\infty}$
satisfying $\Psi_{\alpha}\rho=\mu$, it suffices to choose $\Gamma_{\rho}$,
$\lambda_{\beta}^{\rho}$, $A_{\rho}$, and $\gamma_{\rho}$ such that \eqref{A1-2}
holds and
\begin{gather}
\label{A1-3}  A_{\mu}=\int_0^{\infty} f_{\alpha}(s)^2 ds A_{\rho},\\
\label{A1-4} \gamma_{\mu}=\int_0^{\infty-} f_{\alpha}(s)ds\left( \gamma_{\rho}
+\int_{\mathbb{R}^d} x(1_{\{|f_{\alpha}(s)x|\leq1\}} -1_{\{|x|\leq1\}})
\nu_{\rho}(dx)\right)
\end{gather}
(see Proposition 3.18 of [S]).
This choice is possible, because $\inf_{\beta\in(0,2)} \Gamma(\beta-\alpha)>0$,
$\int_0^{\infty}f_{\alpha}(s)ds=\int_0^{\infty} t^{-\alpha}e^{-t}dt=\Gamma(1-\alpha)$,
$\int_0^{\infty}f_{\alpha}(s)^2ds=\int_0^{\infty} t^{1-\alpha}e^{-t}dt=\Gamma(2-\alpha)$,
and
\begin{align*}
&\int_0^{\infty} f_{\alpha}(s)ds \int_{\mathbb{R}^d} |x|\,|1_{\{|f_{\alpha}(s)x|\leq1\}} 
-1_{\{|x|\leq1\}}|\,\nu_{\rho}(dx)\\
&\qquad=\int_0^{\infty} t^{-\alpha}e^{-t}dt \int_{\mathbb{R}^d} |x|\,|1_{\{|tx|\leq1\}} 
-1_{\{|x|\leq1\}}|\,\nu_{\rho}(dx)\\
&\qquad=\int_0^1 t^{-\alpha}e^{-t}dt \int_{1<|x|\leq 1/t} |x|\,\nu_{\rho}(dx)
+\int_1^{\infty} t^{-\alpha}e^{-t}dt \int_{1/t<|x|\leq 1} |x|\,\nu_{\rho}(dx)\\
&\qquad=\int_{|x|>1} |x|\,\nu_{\rho}(dx) \int_0^{1/|x|}t^{-\alpha}e^{-t}dt
+\int_{|x|\leq 1} |x|\,\nu_{\rho}(dx) \int_{1/|x|}^{\infty}t^{-\alpha}e^{-t}dt<\infty,
\end{align*}
since $\int_0^{1/|x|}t^{-\alpha}e^{-t}dt\sim (1-\alpha)^{-1} |x|^{\alpha-1}$ as
$|x|\to\infty$ and $\int_{1/|x|}^{\infty}t^{-\alpha}e^{-t}dt\sim  |x|^{\alpha}
e^{-1/|x|}$ as $|x|\downarrow0$.  Therefore \eqref{A1-1} is true.
It follows that $\Psi_{\alpha}^n(L_{\infty})=L_{\infty}$ for $n=1,2.\ldots$.
Hence $\mathfrak R_{\infty}(\Psi_{\alpha})\supset L_{\infty}$.  On the other hand,
$\mathfrak R_{\infty}(\Psi_{\alpha})\subset L_{\infty}$ by virtue of Lemma \ref{l3f}.

{\it Case 2} ($0\leq\alpha<1$).  Since $\mathfrak D(\Psi_{\alpha})$ is as in Lemma \ref{l3a},
it follows from Lemma \ref{l3c} that
\[
L_{\infty}\cap\mathfrak D(\Psi_{\alpha})=\begin{cases}
\{\rho\in L_{\infty}\colon \int_{(0,2)} \beta^{-2}\,\Gamma_{\rho}(d\beta)<\infty\},\quad
&\alpha=0,\\
\{\rho\in L_{\infty}^{(\alpha,2)}\colon \int_{(\alpha,2)} (\beta-\alpha)^{-1}
\,\Gamma_{\rho}(d\beta)
<\infty\},\quad&0<\alpha<1. \end{cases}
\]
We have
\begin{equation}\label{A2-1}
\Psi_{\alpha}(L_{\infty}\cap\mathfrak D(\Psi_{\alpha})) =L_{\infty}^{(\alpha,2)},
\end{equation}
where $L_{\infty}^{(0,2)}=L_{\infty}$.  Indeed, if $\rho\in L_{\infty}\cap
\mathfrak D(\Psi_{\alpha})$ and $\mu=\Psi_{\alpha}\rho$, then we have $\mu\in 
L_{\infty}^{(\alpha,2)}$ and \eqref{A1-2}, using $\Gamma(\beta-\alpha)=(\beta-\alpha)^{-1}
\Gamma(\beta-\alpha+1)$ for $0\leq\alpha<1$.  Conversely, if $\mu\in 
L_{\infty}^{(\alpha,2)}$, then we can find $\rho\in L_{\infty}\cap
\mathfrak D(\Psi_{\alpha})$ satisfying $\mu=\Psi_{\alpha}\rho$ in the same way as in
Case 1; when $\alpha=0$, we have $\int_{(0,2)} \beta^{-2}
\Gamma_{\rho}(d\beta)<\infty$ since $\Gamma_{\rho}(d\beta)=\beta(\Gamma(\beta+1))^{-1}
\Gamma_{\mu}(d\beta)$ and $\int_{(0,2)} \beta^{-1}\Gamma_{\mu}(d\beta)<\infty$.
Hence \eqref{A2-1} holds.  Now we have
\begin{equation}\label{A2-2}
\Psi_{\alpha}^n(L_{\infty}\cap\mathfrak D(\Psi_{\alpha}^n))=L_{\infty}^{(\alpha,2)}
\end{equation}
for $n=1,2,\ldots$.  Indeed, it is true for $n=1$ by \eqref{A2-1}
and, if \eqref{A2-2} is true for $n$, then
\begin{align*}
L_{\infty}^{(\alpha,2)}&=\Psi_{\alpha}^n(L_{\infty}\cap\mathfrak D(\Psi_{\alpha}^n))=
\Psi_{\alpha}^n(L_{\infty}^{(\alpha,2)}\cap\mathfrak D(\Psi_{\alpha}^n))\\
&=\Psi_{\alpha}^n(\Psi_{\alpha}(L_{\infty}\cap\mathfrak D(\Psi_{\alpha}))
\cap\mathfrak D(\Psi_{\alpha}^n))\\
&=\Psi_{\alpha}^n(\Psi_{\alpha}(L_{\infty}\cap\mathfrak D(\Psi_{\alpha}^{n+1})))
=\Psi_{\alpha}^{n+1}(L_{\infty}\cap\mathfrak D(\Psi_{\alpha}^{n+1})).
\end{align*}
It follows from \eqref{A2-2} that $L_{\infty}^{(\alpha,2)}\subset
\mathfrak R_{\infty}(\Psi_{\alpha})$.  Next we claim that
\begin{equation}\label{A2-3}
\mathfrak R(\Psi_{\alpha})\cap L_{\infty}\subset L_{\infty}^{(\alpha,2)}.
\end{equation}
Let $\mu\in \mathfrak R(\Psi_{\alpha})\cap L_{\infty}$.  Then $\mu$ has a
radial decomposition $(\lambda_{\mu}(d\xi),r^{-\alpha-1}\,k_{\xi}^{\mu}(r)dr)$
with the property stated in Lemma \ref{l3b}. On the other hand,
\begin{align*}
\nu_{\mu}(B)&=\int_{(0,2)} \Gamma_{\mu}(d\beta)\int_S \lambda_{\beta}^{\mu}(d\xi) 
\int_0^{\infty} 1_B(r\xi) r^{-\beta-1} dr\\
&=\int_S \overline\lambda_{\mu}(d\xi) \int_{(0,2)} \Gamma_{\xi}^{\mu}(d\beta)
\int_0^{\infty} 1_B(r\xi) r^{-\beta-1} dr
\end{align*}
for $B\in\mathcal B(\mathbb{R}^d)$, 
as there are a probability measure $\overline\lambda_{\mu}$ on $S$ and a 
measurable family $\{\Gamma_{\xi}^{\mu}\}$ of measures on $(0,2)$
satisfying $\int_{(0,2)}(\beta^{-1}+(2-\beta)^{-1})\Gamma_{\xi}^{\mu}(d\beta)
=\mathrm{const}$ such that $\Gamma_{\mu}(d\beta)\lambda_{\beta}^{\mu}(d\xi)=
\overline\lambda_{\mu}(d\xi)\Gamma_{\xi}^{\mu}(d\beta)$.  Hence, by the uniqueness
in Proposition 3.1 of [S], there is a positive, finite, measurable
function $c(\xi)$ such that $\lambda_{\mu}(d\xi)=c(\xi)\overline\lambda_{\mu}(d\xi)$
and, for $\lambda_{\mu}$-a.\,e.\ $\xi$, $r^{-\alpha-1}\,k_{\xi}^{\mu}(r)dr
=c(\xi)^{-1}\left( \int_{(0,2)} r^{-\beta-1}\Gamma_{\xi}^{\mu}(d\beta)\right)dr$.
Hence $k_{\xi}^{\mu}(r)=c(\xi)^{-1} \int_{(0,2)} r^{\alpha-\beta}\allowbreak 
\Gamma_{\xi}^{\mu}(d\beta)$, a.\,e.\ $r$.  Since $k_{\xi}^{\mu}(r)$ is
completely monotone, it vanishes as $r$ goes to infinity.  
Hence $\Gamma_{\xi}^{\mu}((0,\alpha])=0$ for $\lambda_{\mu}$-a.\,e.\ $\xi$.
Hence $\Gamma_{\mu}((0,\alpha])=0$, that is, $\mu\in L_{\infty}^{(\alpha,2)}$,
proving \eqref{A2-3}.
Now, using Lemma \ref{l3f}, we obtain $\mathfrak R_{\infty}(\Psi_{\alpha})\subset 
\mathfrak R(\Psi_{\alpha})\cap L_{\infty}\subset L_{\infty}^{(\alpha,2)}$.

{\it Case 3} ($\alpha=1$).  Let us show that  
\begin{equation}\label{A3-1}
\Psi_1(L_{\infty}\cap \mathfrak D(\Psi_1))=L_{\infty}^{(1,2)}\cap
\{\mu\in ID\colon \text{weak mean }0\}.
\end{equation}
Let $\rho\in L_{\infty}\cap \mathfrak D(\Psi_1)$, that is, 
$\rho\in L_{\infty}^{(1,2)}$,
$\int_{(1,2)} (\beta-1)^{-1}\Gamma_{\rho}(d\beta)<\infty$, 
$\int_{\mathbb{R}^d} x\rho(dx)=0$,
and $\lim_{a\to\infty} \int_1^a s^{-1} ds \int_{|x|>s} x\nu_{\rho}(dx)$
exists in $\mathbb{R}^d$.  Let $\mu=\Psi_1\rho$.
Then, as in Case 1, $\mu\in L_{\infty}^{(1,2)}$ and 
\eqref{A1-2} holds with $\alpha=1$.  By Lemma \ref{l3b}, $\mu$ has weak mean $0$.
Conversely, let $\mu\in L_{\infty}^{(1,2)}\cap \{\mu\in ID\colon \text{weak mean }0\}$.
Choose $\rho\in L_{\infty}^{(1,2)}$ such that 
$\Gamma_{\rho}(d\beta)=(\Gamma(\beta-1))^{-1}\Gamma_{\mu}(d\beta)$, 
$\lambda_{\beta}^{\rho}=\lambda_{\beta}^{\mu}$, $A_{\rho}=A_{\mu}$,
and $\gamma_{\rho}=-\int_{|x|>1}x\nu_{\rho}(dx)$ (note that $\int_{(1,2)}(\beta-1)^{-1}
\Gamma_{\rho}(d\beta)<\infty$ and hence $\int_{|x|>1}|x|
\nu_{\rho}(dx)<\infty$ by Lemma \ref{l3c}).  Then $\int_{\mathbb{R}^d} x\rho(dx)=0$ 
(see Lemma 4.3 of [S]).  Since $\mu$ has weak mean, $\int_1^a s^{-1}ds\int_{|x|>s}
x\nu_{\rho}(dx)$ is convergent as $a\to\infty$ by application of
Lemma \ref{l3d} with $b(\beta)=1/\Gamma(\beta)$.
Hence $\rho\in\mathfrak D(\Psi_1)$.  We have $\nu_{\Psi_1 \rho}=\nu_{\mu}$,
$A_{\Psi_1 \rho}=A_{\mu}$, and $\Psi_1\rho$ has weak mean $0$.  Among
distributions $\mu'\in ID$ having $\nu_{\mu'}=\nu_{\mu}$ and 
$A_{\mu'}=A_{\mu}$, only one distribution has weak mean $0$.  Hence
$\Psi_1\rho=\mu$.  This proves \eqref{A3-1}.  We have
\begin{equation}\label{A3-2}
\Psi_1^n(L_{\infty}\cap \mathfrak D(\Psi_1^n))=L_{\infty}^{(1,2)}\cap
\{\mu\in ID\colon \text{weak mean }0\},\qquad n=1,2,\ldots
\end{equation}
from \eqref{A3-1} by the same argument as in Case 2.  Hence
\begin{equation}\label{A3-3}
L_{\infty}^{(1,2)}\cap
\{\mu\in ID\colon \text{weak mean }0\} \subset \mathfrak R_{\infty}(\Psi_1).
\end{equation}
Next
\begin{equation}\label{A3-4}
\mathfrak R(\Psi_1)\cap L_{\infty} \subset L_{\infty}^{(1,2)}\cap
\{\mu\in ID\colon \text{weak mean }0\}.
\end{equation}
Indeed, $\mathfrak R(\Psi_1)\cap L_{\infty} \subset L_{\infty}^{(1,2)}$
by the same argument as in Case 2. Any $\mu\in\mathfrak R(\Psi_1)$ has 
weak mean $0$ by Lemma \ref{l3b}.  Now it follows from 
Lemma \ref{l3f} that
\begin{equation}\label{A3-5}
\mathfrak R_{\infty}(\Psi_1) \subset L_{\infty}^{(1,2)}\cap
\{\mu\in ID\colon \text{weak mean }0\}.
\end{equation}

{\it Case 4} ($1<\alpha<2$).  We show that
\begin{equation}\label{A4-1}
\Psi_{\alpha}(L_{\infty}\cap \mathfrak D(\Psi_{\alpha}))=L_{\infty}^{(\alpha,2)}\cap
\{\mu\in ID\colon \text{mean }0\}.
\end{equation}
Let $\rho\in L_{\infty}\cap \mathfrak D(\Psi_{\alpha})$, that is, $\rho\in
L_{\infty}^{(\alpha,2)}$, $\int_{(\alpha,2)} (\beta-\alpha)^{-1}
\Gamma_{\rho}(d\beta)<\infty$, 
and $\int_{\mathbb{R}^d} x\rho(dx)=0$ (Lemmas \ref{l3a} and \ref{l3c}).  
Let $\mu=\Psi_{\alpha}\rho$. Then $\mu\in L_{\infty}^{(\alpha,2)}$ and 
\eqref{A1-2} holds.  Hence $\int_{\mathbb{R}^d} |x|\mu(dx)<\infty$ by
Lemma \ref{l3c} and $\mu$ has mean $0$ by Lemma \ref{l3b}.
Conversely, if $\mu\in L_{\infty}^{(\alpha,2)}\cap \{\mu\in ID\colon \text{mean }0\}$,
then we can find $\rho\in L_{\infty}\cap \mathfrak D(\Psi_{\alpha})$ satisfying 
$\Psi_{\alpha}\rho=\mu$, similarly to Case 3.  Hence \eqref{A4-1} is true.
It follows that 
\[
\Psi_{\alpha}^n(L_{\infty}\cap \mathfrak D(\Psi_{\alpha}^n))=L_{\infty}^{(\alpha,2)}
\cap\{\mu\in ID\colon \text{mean }0\},\qquad n=1,2,\ldots
\]
similarly to Cases 2 and 3.  Hence
\begin{equation}\label{A4-2}
L_{\infty}^{(\alpha,2)}\cap
\{\mu\in ID\colon \text{mean }0\} \subset \mathfrak R_{\infty}(\Psi_{\alpha}).
\end{equation}
We can also prove
\[
\mathfrak R(\Psi_{\alpha})\cap L_{\infty} \subset L_{\infty}^{(\alpha,2)}\cap
\{\mu\in ID\colon \text{mean }0\}
\]
similarly to Cases 2 and 3.  Hence the reverse inclusion of \eqref{A4-2}
follows from Lemma \ref{l3f}.
\qed

\medskip
\noindent {\it Proof of the part related to 
$\mathfrak R_{\infty}(\bar\Phi_{p,\alpha})$ in Theorem \ref{t1a}}.  We assume $p\geq1$.
Since monotonicity of order $p\in[1,\infty)$ implies monotonicity of order
$1$ (Corollary 2.6 of [S]), it follows from Lemma \ref{l3b} that 
Lemma \ref{l3f} is applicable with $\beta=(-\alpha-1)\lor 0$.  Hence
$\mathfrak R_{\infty}(\bar\Phi_{p,\alpha})\subset L_{\infty}$.  If 
$\rho\in L_{\infty}\cap\mathfrak D(\bar\Phi_{p,\alpha})$ and 
$\bar\Phi_{p,\alpha} \rho=\mu$,  then 
$\rho\in L_{\infty}^{(\alpha,2)}$ (understand that 
$L_{\infty}^{(\alpha,2)}=L_{\infty}$ for $\alpha\leq0$) and, noting that
\begin{align*}
\nu_{\mu}(B)&=\int_0^{\infty} ds\int_{\mathbb{R}^d} 1_B(\bar f_{p,\alpha}(s)x)
\nu_{\rho}(dx)\\
&=\frac1{\Gamma(p)} \int_0^1 t^{-\alpha-1}(1-t)^{p-1}dt\int_{\mathbb{R}^d} 
1_B(tx)\nu_{\rho}(dx)\\
&=\frac1{\Gamma(p)}\int_0^1 t^{-\alpha-1}(1-t)^{p-1}dt\int_{(0,2)} 
\Gamma_{\rho}(d\beta)
\int_S \lambda_{\beta}^{\rho}(d\xi)\int_0^{\infty} 1_B(tr\xi)r^{-\beta-1} dr\\
&=\int_{(0,2)} \frac{\Gamma(\beta-\alpha)}{\Gamma(\beta-\alpha+p)} 
\Gamma_{\rho}(d\beta) 
\int_S \lambda_{\beta}^{\rho}(d\xi)\int_0^{\infty} 1_B(u\xi)u^{-\beta-1} du
\end{align*}
and recalling Lemmas \ref{l3a} and \ref{l3c},
we obtain $\mu\in L_{\infty}^{(\alpha,2)}$ with
\begin{equation}\label{B-1}
\Gamma_{\mu}(d\beta)=\frac{\Gamma(\beta-\alpha)}{\Gamma(\beta-\alpha+p)} 
\Gamma_{\rho}(d\beta)\quad\text{and}\quad
\lambda_{\beta}^{\mu}=\lambda_{\beta}^{\rho}.
\end{equation}

Now the proof of assertions (i), (ii), and (iv) can be given in parallel to
the corresponding assertions for $\Psi_{\alpha}$.  Note that, if $-\infty<\alpha<1$, 
then
\[
\int_0^{\infty} \bar f_{p,\alpha}(s)ds \int_{\mathbb{R}^d} |x|
\,|1_{\{|\bar f_{p,\alpha}(s)x|\leq1\}} 
-1_{\{|x|\leq1\}}|\,\nu_{\rho}(dx)<\infty
\]
similarly.  We also use the fact 
that $k_{\xi}^{\mu}(r)$ vanishes at infinity if it is monotone
of order $p\in[1,\infty)$.

For assertion (iii) in the case $\alpha=1$,  we have to find another way, 
as Lemma \ref{l3d} is not applicable if $\beta>1$.  Let us show
\begin{equation}\label{B-2}
\bar\Phi_{p,1}(L_{\infty}\cap \mathfrak D(\bar\Phi_{p,1}))
=L_{\infty}^{(1,2)}\cap\{\mu\in ID
\colon \text{weak mean }0\}.
\end{equation}
Suppose that $\rho\in L_{\infty}\cap \mathfrak D(\bar\Phi_{p,1})$ and 
$\bar\Phi_{p,1}\rho=\mu$.
Then $\rho\in L_{\infty}^{(1,2)}$, $\int_{(1,2)} (\beta-1)^{-1}
\Gamma_{\rho}(d\beta)<\infty$,
$\mu\in L_{\infty}^{(1,2)}$ with \eqref{B-1}, and $\mu$ has weak mean $0$ by
Lemma \ref{l3b}.  Conversely, suppose that $\mu\in L_{\infty}^{(1,2)}$ with
weak mean $0$.  As in [S], let $\mathfrak M^L$ be the class of L\'evy measures
of  infinitely divisible distributions on $\mathbb{R}^d$ and let $\bar\Phi_{p,1}^L$ 
be the transformation of L\'evy measures associated with the mapping $\bar\Phi_{p,1}$.
Define $\Gamma_0(d\beta)=\frac{\Gamma(\beta-1+p)}{\Gamma(\beta-1)} 
\Gamma_{\mu}(d\beta)$.  Then
$\int_{(1,2)} (2-\beta)^{-1} \Gamma_0(d\beta)<\infty$.  Define
\[
\nu_0(B)=\int_{(1,2)} \Gamma_0(d\beta)\int_S \lambda_{\beta}^{\mu}(d\xi)
\int_0^{\infty} 1_B(r\xi) r^{-\beta-1}dr
\]
for $B\in\mathcal B(\mathbb{R}^d)$.  We have $\nu_0\in\mathfrak M^L$.  We see
\begin{align*}
\nu_{\mu}(B)&=\int_{(1,2)} \frac{\Gamma(\beta-1)}{\Gamma(\beta-1+p)} 
\Gamma_0(d\beta) 
\int_S \lambda_{\beta}^{\mu}(d\xi)\int_0^{\infty} 1_B(u\xi)u^{-\beta-1} du\\
&=\int_0^{\infty} ds\int_{\mathbb{R}^d} 1_B(\bar f_{p,1}(s)x)\nu_0(dx)
\end{align*}
from the calculation above.   Since 
$\nu_{\mu}\in\mathfrak M^L$, we have $\nu_0\in \mathfrak D(\bar\Phi_{p,1}^L)$ and
$\bar\Phi_{p,1}^L \nu_0=\nu_{\mu}$.  Then it follows from Theorem 4.10
of [S] that $\nu_{\mu}$ has a radial decomposition 
$(\lambda_{\mu}(d\xi),u^{-2}k_{\xi}^{\mu}(u)du)$ such that $k_{\xi}^{\mu}(u)$
is measurable in $(\xi,u)$ and, for $\lambda_{\mu}$-a.\,e.\ $\xi$,
monotone of order $p$ in $u\in\mathbb{R}_{+}^{\circ}$.  Hence 
$\mu\in\mathfrak R(\bar\Phi_{p,1})$
from Lemma \ref{l3b}.  Since $\bar\Phi_{p,1}^L \nu_0=\nu_{\mu}$
and $\bar\Phi_{p,1}^L$ is one-to-one (Theorem 4.9 of [S]), we
have $\mu=\bar\Phi_{p,1}\rho$ for some $\rho\in\mathfrak D(\bar\Phi_{p,1})$ 
with $\nu_{\rho}=\nu_0$.
It follows that $\rho\in L_{\infty}$.
This finishes the proof of \eqref{B-2}.  
Now we can show \eqref{A3-2}--\eqref{A3-5} with $\bar\Phi_{p,1}$ in 
place of $\Psi_1$ similarly to Case 3 in the preceding proof.
\qed

\medskip
\noindent {\it Proof of the part related to $\mathfrak R_{\infty}(\Lambda_{q,\alpha})$
in Theorem \ref{t1a}}.  
Since we have Proposition \ref{p2c}, it remains only to consider $\Lambda_{1,1}$.
But the assertion for $\mathfrak R_{\infty}(\Lambda_{1,1})$ is obviously true, since
$\Lambda_{1,1}=\bar\Phi_{1,1}$.
\qed


\section{Some examples of $\mathfrak R_{\infty}(\Phi_f)$}

We present some examples of $\Phi_f$ for which the class $\mathfrak R_{\infty}(\Phi_f)$ 
is different from those appearing in Theorem \ref{t1a}.

Define $T_a$, the dilation by $a\in \mathbb{R}\setminus\{0\}$, as $(T_a\mu)(B)
=\int 1_B(ax)\mu(dx)=\mu((1/a)B)$, $B\in\mathcal B(\mathbb{R}^d)$, for measures on 
$\mathbb{R}^d$.  Define $P_t$, the raising to the convolution power $t>0$, in such a way
that, for $\mu\in ID$, $P_t\mu$ is an  infinitely divisible distribution with
characteristic function $\widehat{P_t\mu}(z)=
\widehat\mu(z)^t$.  The mappings $T_a$ (restricted to $ID$),
$P_t$, and $\Phi_f$ are commutative with each other.  
A measure $\mu$ on $\mathbb{R}^d$ is called symmetric if $T_{-1}\mu=\mu$.  
A distribution $\mu$ on $\mathbb{R}^d$ is called shifted symmetric if 
$\mu=\rho*\delta_{\gamma}$ with some symmetric
distribution $\rho$ and some $\delta$-distribution $\delta_{\gamma}$.
Let $ID_{\mathrm{sym}}=ID_{\mathrm{sym}}(\mathbb{R}^d)$  
[resp. $ID_{\mathrm{sym}}^{\mathrm{shift}}
=ID_{\mathrm{sym}}^{\mathrm{shift}}(\mathbb{R}^d)$] 
denote the class of symmetric [resp.\ shifted symmetric]  infinitely divisible 
distributions on $\mathbb{R}^d$.  

\begin{ex}\label{e4a}
Let $f(s)=b1_{[0,a]}(s)-b1_{(a,2a]}(s)$ with $a>0$ and $b\neq0$.  Then 
$\mathfrak R_{\infty}(\Phi_f)=ID_{\mathrm{sym}}$.

Indeed, for $\rho\in ID$,
\[
C_{\Phi_f\rho}(z)=\int_0^a C_{\rho}(bz)ds+\int_a^{2a} C_{\rho}(-bz)ds
=aC_{\rho}(bz)+aC_{\rho}(-bz)=C_{P_a T_b(\rho*T_{-1}\rho)}(z)
\]
for $z\in\mathbb{R}^d$, 
and hence $\Phi_f\rho=P_a T_b(\rho*T_{-1}\rho)$. 
Define $U\rho=P_{1/2}\rho*T_{-1}P_{1/2}\rho$.  Then $U\rho\in ID_{\mathrm{sym}}$
for any $\rho\in ID$.  If $\rho\in ID_{\mathrm{sym}}$, then $U\rho=\rho$.  Hence
$U^n\rho=U\rho$ for $n=1,2,\ldots$.  Since $\Phi_f=P_a T_b P_2 U
=P_{2a} T_b U$, we have $\Phi_f^n=P_{2a}^n T_b^n U=U P_{2a}^n T_b^n$
and $U=\Phi_f^n P_{1/(2a)}^n T_{1/b}^n$.  Hence $\mathfrak R_{\infty}(\Phi_f)
=\mathfrak R(U)=ID_{\mathrm{sym}}$.
\end{ex}

\begin{ex}\label{e4b}
Let $f(s)=b1_{[0,a]}(s)-b1_{(a,a+c]}(s)$ with $a>0$, $c>0$, $a\neq c$, 
and $b\neq0$.  Then 
$\mathfrak R_{\infty}(\Phi_f)=ID_{\mathrm{sym}}^{\mathrm{shift}}$.

To see this, notice that
\[
C_{\Phi_f\rho}(z)=a C_{\rho}(bz)+c C_{\rho}(-bz)=(a+c)(a_1 C_{T_b\rho}(z)
+(1-a_1) C_{T_b\rho}(-z))
\]
for $\rho\in ID$, where $a_1=a/(a+c)$.
That is, $\Phi_f\rho=P_{a+c}T_b(P_{a_1}\rho*P_{1-a_1}T_{-1}\rho)$.  Let us define
$V\rho=P_{a_1}\rho*P_{1-a_1}T_{-1}\rho$.  Note that 
$V$ is the stochastic integral mapping $\Phi_f$ in the case
$a+c=1$ and $b=1$.   We have
\begin{equation}\label{e4b1}
V^n\rho=P_{a_n}\rho*P_{1-a_n}T_{-1}\rho
\end{equation}
for $n=1,2,\ldots$, where $a_n$ is given by $a_n=1-a_1+a_{n-1}(2a_1-1)$.
Indeed, if \eqref{e4b1} is true for $n$, then it is true for $n+1$ in
place of $n$, since
\begin{align*}
V^{n+1}\rho&=P_{a_n}V\rho*P_{1-a_n}T_{-1}V\rho=P_{a_n}V\rho*P_{1-a_n}VT_{-1}\rho\\
&=P_{a_n}(P_{a_1}\rho*P_{1-a_1}T_{-1}\rho)*P_{1-a_n}(P_{a_1}T_{-1}\rho*P_{1-a_1}\rho)
\\
&=P_{a_n a_1+(1-a_n)(1-a_1)}\rho*P_{a_n(1-a_1)+(1-a_n)a_1}T_{-1}\rho\\
&=P_{a_{n+1}}\rho * P_{1-a_{n+1}}T_{-1}\rho.
\end{align*}
We see that $0<a_n<1$ for all $n$.
We have $\Phi_f^n=P_{a+c}^nT_b^n V^n=V^n P_{a+c}^nT_b^n$ and 
$V^n=P_{1/(a+c)}^n T_{1/b}^n \Phi_f^n=\Phi_f^n P_{1/(a+c)}^n T_{1/b}^n$.
Therefore $\mathfrak R(\Phi_f^n)=\mathfrak R(V^n)$ and hence 
$\mathfrak R_{\infty}(\Phi_f)=\mathfrak R_{\infty}(V)$.
Next let us show that
\begin{equation}\label{e4b2}
\mathfrak R_{\infty}(V)=ID_{\mathrm{sym}}^{\mathrm{shift}}.
\end{equation}
If $\rho\in ID_{\mathrm{sym}}$, then $V\rho=\rho$.  Hence $ID_{\mathrm{sym}}
\subset\mathfrak R_{\infty}(V)$.
If $\rho=\delta_{\gamma}$, then $V\rho=\delta_{a_1\gamma}*\delta_{-(1-a_1)\gamma}
=\delta_{(2a_1-1)\gamma}$.
Now $\delta_{\gamma}=V\delta_{(1/(2a_1-1))\gamma}$, since $a_1\neq 1/2$.
Hence all $\delta$-distributions are in $\mathfrak R(V^n)$ and hence in 
$\mathfrak R_{\infty}(V)$.
Since $\mathfrak R_{\infty}(V)$ is closed under convolution, we obtain 
$ID_{\mathrm{sym}}^{\mathrm{shift}}\subset \mathfrak R_{\infty}(V)$.  
To show the converse, assume that
$\mu\in \mathfrak R_{\infty}(V)$.  Then $\mu=V^n\rho_n$ for some $\rho_n\in ID$.
It follows from \eqref{e4b1} that $\nu_{\mu}=a_n \nu_{\rho_n} +(1-a_n)T_{-1}\nu_{
\rho_n}$.  Let $\sigma_n\in ID$ be such that $(A_{\sigma_n},\nu_{\sigma_n},
\gamma_{\sigma_n})=(0,\nu_{\rho_n},0)$.  
It follows from $a_n=1-a_1+a_{n-1}(2a_1-1)$ and from $0<a_n<1$ that $a_n\to1/2$
as $n\to\infty$.  Hence $a_n>1/3$ for all large $n$.  We see that the
set $\{\sigma_n\colon n=1,2,\ldots\}$ is precompact,
since $\nu_{\sigma_n}\leq a_n^{-1}\nu_{\mu}\leq 3\nu_{\mu}$ for all large $n$.
Thus we can choose a subsequence $\{\sigma_{n_k}\}$
convergent to some $\mu'\in ID$.  Since $\int\varphi(x) \nu_{\sigma_{n_k}}(dx)\to
\int\varphi(x)\nu_{\mu'}(dx)$ for any bounded continuous function $\varphi$ which 
vanishes on a neighborhood of the origin and since $a_n\to1/2$,  we obtain
$\nu_{\mu}=(1/2)\nu_{\mu'}+(1/2)T_{-1}\nu_{\mu'}$.  Hence $\nu_{\mu}$ is
symmetric. Hence $\mu*\delta_{-\gamma_{\mu}}$ is symmetric.  It follows that
$\mu\in ID_{\mathrm{sym}}^{\mathrm{shift}}$.  This proves \eqref{e4b2} and therefore 
$\mathfrak R_{\infty}(\Phi_f)=ID_{\mathrm{sym}}^{\mathrm{shift}}$.
\end{ex}

\begin{ex}\label{e4c}
Let $\alpha<0$.  Let $h(s)$ be one of $f_{\alpha}(s)$, $\bar f_{p,\alpha}(s)$, and 
$l_{q,\alpha}(s)$ ($p\geq1$, $q>0$).  Let $s_0=\sup\{s\colon h(s)>0\}$.  Then
$0<s_0<\infty$.  Define 
\[
f(s)=\begin{cases} h(s), \quad & 0\leq s\leq s_0,\\
-h(2s_0-s), \quad & s_0<s\leq 2s_0,\\
0, \quad & s>2s_0.
\end{cases}
\]
Then $\mathfrak R_{\infty}(\Phi_f)=L_{\infty}\cap ID_{\mathrm{sym}}$.

Proof is as follows.  First, recall that $\mathfrak D(\Phi_f)=\mathfrak D(\Phi_h)
=ID$.  We have, for $\rho\in ID$,
\begin{align*}
C_{\Phi_f\rho}(z)&=\int_0^{s_0} C_{\rho}(h(s)z)ds+\int_{s_0}^{2s_0} C_{\rho}
(-h(2s_0-s)z)ds\\
&=\int_0^{s_0} C_{\rho}(h(s)z)ds+\int_0^{s_0} C_{\rho}
(-h(s)z)ds\\
&=C_{\Phi_h\rho}(z)+C_{\Phi_h T_{-1}\rho}(z).
\end{align*}
It follows that $\Phi_f\rho=\Phi_h(\rho*T_{-1}\rho)=\Phi_hP_{2} U\rho
=UP_{2} \Phi_h\rho$, where $U$ is the mapping used in Example \ref{e4a}.
It follows that $\Phi_f^n=\Phi_h^n P_{2}^n U=U P_{2}^n \Phi_h^n$
for $n=1,2,\ldots$.  Hence $\mathfrak R(\Phi_f^n)\subset \mathfrak R(\Phi_h^n)\cap 
ID_{\mathrm{sym}}$.  Conversely, assume that $\rho\in \mathfrak R(\Phi_h^n)\cap 
ID_{\mathrm{sym}}$.  Then $\mu=\Phi_h^n \rho$ for some $\rho$ and 
$T_{-1}\mu=\Phi_h^n T_{-1}\rho$.  Since $\Phi_h$ is one-to-one (see 
[S]), we have $\rho=T_{-1}\rho$. Hence $\Phi_f^n\rho=\Phi_h^n P_{2}^n U
\rho=\Phi_h^n P_{2}^n \rho=P_{2}^n\mu$ and thus 
$\mu=\Phi_f^n P_{1/2}^n \rho \in \mathfrak R(\Phi_f^n)$.  In conclusion,
$\mathfrak R(\Phi_f^n)=\mathfrak R(\Phi_h^n)\cap ID_{\mathrm{sym}}$ and hence
$\mathfrak R_{\infty}(\Phi_f)=\mathfrak R_{\infty}(\Phi_h)\cap ID_{\mathrm{sym}}
=L_{\infty}\cap ID_{\mathrm{sym}}$.
\end{ex}

\begin{ex}\label{e4d}
Let $h(s)$ and $s_0$ be as in Example \ref{e4c}.  Define 
\[
f(s)=\begin{cases} h(s_0-s), \quad & 0\leq s\leq s_0,\\
h(s-s_0), \quad & s_0<s\leq 2s_0,\\
-h(3s_0-s), \quad & 2s_0<s\leq 3s_0,\\
0, \quad & s>3s_0.
\end{cases}
\]
Then $\mathfrak R_{\infty}(\Phi_f)=L_{\infty}\cap ID_{\mathrm{sym}}^{\mathrm{shift}}$.

To see this, notice that 
\begin{align*}
C_{\Phi_f\rho}(z)&=\int_0^{s_0} C_{\rho}(h(s_0-s)z)ds+\int_{s_0}^{2s_0} C_{\rho}
(h(s-s_0)z)ds\\
&\quad +\int_{2s_0}^{3s_0} C_{\rho}(-h(3s_0-s)z)ds\\
&=\int_0^{s_0} C_{\rho}(h(s)z)ds+\int_0^{s_0} C_{\rho}(h(s)z)ds+
\int_0^{s_0} C_{\rho}(-h(s)z)ds\\
&=2C_{\Phi_h\rho}(z)+C_{\Phi_h \rho}(-z)\\
&=3(\tfrac23 C_{\Phi_h\rho}(z) +\tfrac13 C_{\Phi_h\rho}(-z)).
\end{align*}
Hence $\Phi_f\rho=P_3 V\Phi_h\rho$, where $V\rho=P_{2/3}\rho*P_{1/3}T_{-1}\rho$.
This mapping $V$ is a special case of $V$ in Example \ref{e4b} with
$a_1=2/3$.  Hence \eqref{e4b1} holds with $a_n=2^{-1}(1+3^{-n})$ and
$1-a_n=2^{-1}(1-3^{-n})$.  Now $\Phi_f^n=P_3^n V^n \Phi_h^n=\Phi_h^n
P_3^n V^n=V^n P_3^n \Phi_h^n$.  Hence $\mathfrak R(\Phi_f^n)\subset 
\mathfrak R(\Phi_h^n)\cap \mathfrak R(V^n)$.  It follows that 
$\mathfrak R_{\infty}(\Phi_f)\subset 
\mathfrak R_{\infty}(\Phi_h)\cap \mathfrak R_{\infty}(V)=L_{\infty}\cap 
ID_{\mathrm{sym}}^{\mathrm{shift}}$ from Theorem \ref{t1a} and \eqref{e4b2}.
Let us also show the converse inclusion $L_{\infty}\cap 
ID_{\mathrm{sym}}^{\mathrm{shift}}\subset \mathfrak R_{\infty}(\Phi_f)$.  It is enough
to show 
\begin{equation}\label{e4d1}
\mathfrak R(\Phi_h^n)\cap ID_{\mathrm{sym}}^{\mathrm{shift}}\subset \mathfrak R(\Phi_f^n).
\end{equation}
For any $\gamma\in\mathbb{R}^d$ we have
\[
C_{\Phi_h \delta_{\gamma}}(z)=\int_0^{s_0} C_{\delta_{\gamma}}(h(s)z)ds=i\int_0^{s_0}
\langle\gamma,h(s)z\rangle ds=ic\langle\gamma,z\rangle =C_{\delta_{c\gamma}}(z),
\]
where $c=\int_0^{s_0} h(s)ds>0$.  That is, $\Phi_h\delta_{\gamma}=\delta_{c\gamma}$.
Hence $\Phi_f\delta_{\gamma}=P_3\Phi_h V\delta_{\gamma}=P_3\Phi_h(\delta_{(2/3)\gamma}
*\delta_{-(1/3)\gamma})=\Phi_h\delta_{\gamma}=\delta_{c\gamma}$. Hence 
$\Phi_f^n \delta_{\gamma}
=\delta_{c^n\gamma}$ and $\delta_{\gamma}=\Phi_f^n \delta_{c^{-n}\gamma}$. Hence all
$\delta$-distributions are in $\mathfrak R(\Phi_f^n)$.  Similarly all
$\delta$-distributions are in $\mathfrak R(\Phi_h^n)$. 
Let $\mu\in \mathfrak R(\Phi_h^n)\cap ID_{\mathrm{sym}}^{\mathrm{shift}}$.  Then
$\mu*\delta_{\gamma}\in \mathfrak R(\Phi_h^n)\cap ID_{\mathrm{sym}}$ for some $\gamma$.
Letting $\mu'=\mu*\delta_{\gamma}$, we have $\mu'=\Phi_h^n\rho'$ for some $\rho'$.  
Since $\mu'=T_{-1}\mu'=\Phi_h^n T_{-1}\rho'$, we have $\rho'=T_{-1}\rho'$ from
the one-to-one property of $\Phi_h$.  Thus $V^n\rho'=\rho'$ and $\Phi_f^n\rho'=
\Phi_h^n P_3^n\rho'=P_s^n\mu'$.  Hence $\mu'=P_{1/3}^n \Phi_f^n \rho'
=\Phi_f^n P_{1/3}^n \rho'\in \mathfrak R(\Phi_f^n)$.  It follows that 
$\mu=\mu'*\delta_{-\gamma}\in \mathfrak R(\Phi_f^n)$. This proves \eqref{e4d1}.
Hence $\mathfrak R_{\infty}(\Phi_f)=L_{\infty}\cap ID_{\mathrm{sym}}^{\mathrm{shift}}$.
\end{ex}

\begin{ex}\label{e4e}
Let $b>1$.  Let $f(s)=b1_{[0,1]}(s)+1_{(1,2]}(s)$.  Let $L_{\infty}(b)$
be the class mentioned in Section 2.  
Then $L_{\infty}(b)
\subset \mathfrak R_{\infty}(\Phi_f)\subsetneqq ID$.  We do not know whether 
$\mathfrak R_{\infty}(\Phi_f)$ equals $L_{\infty}(b)$. 

Let us show that $L_{\infty}(b)
\subset \mathfrak R_{\infty}(\Phi_f)$.  For $0<\alpha\leq2$ define 
$\mathfrak S_{\alpha}(b)=\mathfrak S_{\alpha}(b,\mathbb{R}^d)$ as follows: 
$\rho\in\mathfrak S_{\alpha}(b)$
if and only if $\rho$ is a $\delta$-distribution or a non-trivial $\alpha$-semi-stable 
distribution with $b$ as a span, that is,
\[
\mathfrak S_{\alpha}(b)=\{\rho\in ID\colon P_{b^\alpha}\rho=T_b\rho*\delta_{\gamma}
\text{ for some }\gamma \in\mathbb{R}^d\}.
\]
We have $C_{\Phi_f\rho}(z)=C_{\rho}(bz)+C_{\rho}(z)$ for $\rho\in ID$, that
is, $\Phi_f\rho=T_b\rho*\rho$.  If $\rho\in\mathfrak S_{\alpha}(b)$ with 
$P_{b^\alpha}\rho=T_b\rho*\delta_{\gamma}$, then $\mu=\Phi_f\rho$ satisfies
$\mu=T_b\rho*\rho=P_{b^{\alpha}}\rho*\delta_{-\gamma}*\rho=P_{b^\alpha+1}
\rho*\delta_{-\gamma}$ and $\mu\in\mathfrak S_{\alpha}(b)$.
If $\mu\in\mathfrak S_{\alpha}(b)$ with 
$P_{b^\alpha}\mu=T_b\mu*\delta_{\gamma'}$, then $\mu=\Phi_f\rho$ for 
$\rho=P_{1/(b^\alpha +1)}
(\mu*\delta_{(1/(b+1))\gamma'})\in \mathfrak S_{\alpha}(b)$.  Therefore 
$\Phi_f(\mathfrak S_{\alpha}(b))=\mathfrak S_{\alpha}(b)$.  Hence  
$\mathfrak S_{\alpha}(b)\subset \mathfrak R(\Phi_f^n)$  for $0<\alpha\leq 2$ and
$n=1,2,\ldots$.
It follows from Proposition 3.2 of Maejima and Sato (2009) that $\mathfrak R(\Phi_f^n)$
is closed under convolution and weak convergence. Hence
$L_{\infty}(b)\subset \mathfrak R(\Phi_f^n)$ and thus 
$L_{\infty}(b)\subset \mathfrak R_{\infty}(\Phi_f)$.
In order to show $\mathfrak R_{\infty}(\Phi_f)\subsetneqq ID$, let $\mu$
be such that $\nu_{\mu}=\delta_a$ with $a\neq 0$.  Suppose that 
$\mu=\Phi_f\rho$ for some $\rho\in ID$.  Then $\nu_{\mu}=T_b\nu_{\rho}
+\nu_{\rho}$.  If $\nu_{\rho}\neq0$, then the support of $\nu_{\rho}$
contains at least one point $a'\neq0$ and hence the support of 
$\nu_{\mu}$ contains at least two points $\{a',ba'\}$, which is
absurd.  If $\nu_{\rho}=0$, then $\nu_{\mu}=0$, which is also absurd.
Therefore $\mu\not\in\mathfrak R(\Phi_f)$ and hence 
$\mu\not\in\mathfrak R_{\infty}(\Phi_f)$.
\end{ex}


\section{Concluding remarks}

The limit class $\mathfrak R_{\infty}(\Phi_f)$ is not known in many cases.
For instance it is not known for the following choices of $f(s)$:
$l_{q,1}(s)$ with $q\in(0,1)\cup(1,\infty)$ in [S];
$\bar f_{p,\alpha}(s)$ with $p\in(0,1)$ and $\alpha\in(-\infty,2)$ in [S];
$\cos (2^{-1}\pi s)$ in Maejima et al.\ (2011a); $e^{-s}\,1_{[0,c]}(s)$ with 
$c\in(0,\infty)$ in Pedersen and Sato (2005); $G_{\alpha,\beta}^* (s)$ with
$\alpha\in[1,2)$ and $\beta>0$ satisfying $\alpha=1+n\beta$ for some $n=0,1,\ldots$ 
in Maejima and Ueda (2010b). Another instance is $\Phi_f=\Upsilon^{\alpha}$ with 
$\alpha\in(0,1)$ related to the Mittag-Leffler function, introduced
in Barndorff-Nielsen and Thorbj\o rnsen (2006).

Consider, as in Sato (2007), a stochastic integral mapping
\[
\Phi_f\,\rho=\mathcal L\left(\int_{0+}^a f(s)dX_s^{(\rho)}\right)
\]
with $0<a<\infty$ for a function $f(s)$ locally square-integrable on
the interval $(0,a]$ and study $\mathfrak R_{\infty}(\Phi_f)=\bigcap_{n=1}^{\infty} 
\mathfrak R(\Phi_f^n)$. Under appropriate choices of $f$ we obtain 
$\mathfrak R_{\infty}(\Phi_f)$ equal to $L_{\infty}^{(0,\alpha)}\cap ID_0$ with
$\alpha\in(1,2)$, $L_{\infty}^{(0,\alpha)}\cap ID_0\cap\{\mu\in ID\colon 
\mu\text{ has drift }0\}$ with $\alpha\in(0,1)$, or a certain subclass of 
$L_{\infty}^{(0,1)}\cap ID_0$.  This will be shown in a
forthcoming paper.

It is an interesting problem what other classes can appear as 
$\mathfrak R_{\infty}(\Phi_f)$.


\begin{thebibliography}{22}

\bibitem{ALM10} T. Aoyama, A. Lindner and M. Maejima.  
A new family of mappings of  infinitely divisible distributions related to the
Goldie--Steutel--Bondesson class. {\it Elect.\ J. Probab.\/} {\bf 15}, 
1119--1142 (2010).

\bibitem{AMR08} T. Aoyama, M. Maejima and J. Rosi\'nski. 
A subclass of type $G$ selfdecomposable distributions. {\it J. Theoret.\ 
Probab.\/} {\bf 21}, 14--34 (2008).

\bibitem{ABP10} O. Arizmendi, O.\,E. Barndorff-Nielsen and 
V. P\'{e}rez-Abreu. On free and classical type $G$ laws.
\emph{Braz.\ J. Prob.\ Stat.\/}  {\bf 24}, 106--127   (2010).

\bibitem{BMS06} O.\,E. Barndorff-Nielsen, M. Maejima and K. Sato.
Some classes of infinitely divisible distributions admitting
stochastic integral representations. \emph{Bernoulli\/} {\bf 12}, 1--33 (2006).

\bibitem{BT06} O.\,E. Barndorff-Nielsen and S. Thorbj\o rnsen.  Regularising
mappings of L\'evy measures. \emph{Stoch.\ Proc.\ Appl.\/}  {\bf 116}, 423--446 (2006). 

\bibitem{G07} B. Grigelionis. Extended Thorin classes and stochastic
integrals. \emph{Liet.\ Mat.\ Rink.\/}  {\bf 47}, 497--503 (2007).

\bibitem{IMU10} K. Ichifuji, M. Maejima and Y. Ueda.  Fixed points of
mappings of  infinitely divisible distributions on $\mathbb{R}^d$. 
\emph{Statist.\ Probab.\ 
Lett.}  {\bf80}, 1320--1328 (2010).

\bibitem{J83} Z.\,J. Jurek. The class $L_m(Q)$ of probability measures on
Banach spaces. \emph{Bull.\ Polish Acad.\ Sci.\ Math.\/} {\bf 31}, 51--62 (1983).

\bibitem{J85} Z.\,J. Jurek. Relations between the $s$-selfdecomposable and 
selfdecomposable measures. \emph{Ann.\ Probab.\/} {\bf 13}, 592--608 (1985).

\bibitem{J88} Z.\,J. Jurek. Random integral representations for classes of
limit distributions similar to L\'evy class $L_0$. \emph{Probab.\ Theory
Relat.\ Fields} {\bf 78}, 473--490 (1988).

\bibitem{J89} Z.\,J. Jurek. Random integral representations for classes of
limit distributions similar to L\'evy class $L_0$, II. \emph{Nagoya Math.\ J.\/} 
{\bf 114}, 53--64 (1989).

\bibitem{J04} Z.\,J. Jurek. The random integral representation
hypothesis revisited: new classes of $s$-selfdecomposable laws. {\it Abstract
and Applied Analysis}, Proc.\ Intern.\ Conf., Hanoi, 2002, World Scientific, 
pp.\,479--498 (2004).

\bibitem{MMS10}  M. Maejima, M. Matsui and M. Suzuki. Classes of
infinitely divisible distributions on $\mathbb{R}^d$ related to the class of 
selfdecomposable distributions.  \emph{Tokyo J. Math.\/} (2010). To appear.

\bibitem{MN09} M. Maejima and G. Nakahara. A note on new classes of
infinitely divisible distributions on $\mathbb{R}^d$.  
\emph{Elect.\ Comm.\ Probab.\/} {\bf 14}, 358--371 (2009).

\bibitem{MPS10a} M. Maejima, V. P\'erez-Abreu and K. Sato.  
A class of multivariate infinitely divisible distributions 
related to arcsine density.  \emph{Bernoulli\/} (2011a). To appear.

\bibitem{MPS10b} M. Maejima, V. P\'erez-Abreu and K. Sato.  
Non-commutative relations of fractional integral transformations 
and Upsilon transformations applied to L\'{e}vy measures (2011b). Preprint.

\bibitem{MS09} M. Maejima and K. Sato. The limits of nested subclasses
of several classes of infinitely divisible distributions are identical with
the closure of the class of stable distributions. \emph{Probab.\ Theory
Relat.\ Fields} {\bf 145}, 119--142 (2009).

\bibitem{MSW00} M. Maejima, K. Sato and T. Watanabe. Completely operator
semi-selfdecomposable distributions. \emph{Tokyo J. Math.} {\bf 23}, 235--253 
(2000).

\bibitem{MU09a} M. Maejima and Y. Ueda. Stochastic integral characterizations
of semi-selfdecomposable distributions and related Ornstein--Uhlenbeck type
processes. \emph{Comm.\ Stoch.\ Anal.}  {\bf 3}, 349--367 (2009a).

\bibitem{MU10} M. Maejima and Y. Ueda. The relation between 
$\lim_{m\to\infty} \mathfrak R(\Psi_{\alpha}^{m+1})$ and  
$\lim_{m\to\infty} \mathfrak R(\Phi_{\alpha}^{m+1})$ (2009b).  
Private Communication.

\bibitem{MU09b} M. Maejima and Y. Ueda. Nested subclasses of the class of
$\alpha$-selfdecomposable distributions. {\it Tokyo J. Math.} (2010a). To appear.

\bibitem{MU09c} M. Maejima and Y. Ueda. Compositions of mappings of 
infinitely divisible distributions with applications to finding the limits
of some nested subclasses.  {\it Elect.\ Comm.\ Probab.}  {\bf 15}, 227--239 (2010b).

\bibitem{MU09d} M. Maejima and Y. Ueda.  $\alpha$-selfdecomposable distributions
and related Ornstein--Uhlenbeck type processes. {\it Stoch.\ Proc.\ Appl.} 
{\bf 120}, 2363--2389 (2010c).

\bibitem{PS05} J. Pedersen and K. Sato. The class of distributions of
periodic Ornstein--Uhlenbeck proccesses driven by L\'evy processes. {\it J.
Theor.\ Probab.} {\bf 18}, 209--235 (2005). 

\bibitem{RS03} A. Rocha-Arteaga and K. Sato. \emph{Topics in Infinitely Divisible 
Distributions and L\'evy Processes\/}. Aportaciones Matem\'aticas, 
Investigaci\'on 17, Sociedad Matem\'atica Mexicana, M\'exico (2003).

\bibitem{S80} K. Sato. Class $L$ of multivariate distributions and
its subclasses. \emph{J. Multivar.\ Anal.} {\bf 10}, 207--232 (1980).

\bibitem{S99} K. Sato. 
{\it L\'evy Processes and Infinitely Divisible Distributions},
Cambridge University Press, Cambridge (1999).

\bibitem{S06b} K. Sato. Two families of improper stochastic integrals
with respect to L\'{e}vy processes. \emph{ALEA Lat.\ Am.\ J.\ Probab.\
Math.\ Statist.} {\bf 1}, 47--87 (2006).

\bibitem{S07} K. Sato. Transformations of  infinitely divisible 
distributions via improper stochastic integrals. 
\emph{ALEA Lat.\ Am.\ J.\ Probab.\ Math.\ Statist.} {\bf 3}, 67--110 (2007).

\bibitem{S10} K. Sato. Fractional integrals and extensions of
selfdecomposability.  \emph{Lecture Notes in Mathematics} (Springer), 
{\bf 2001}, L\'{e}vy Matters I, 1--91 (2010).

\bibitem{U73}  K. Urbanik. Limit laws for sequences of normed sums
satisfying some stability conditions. In: {\it Multivariate Analysis--III} 
(ed. P. R. Krishnaiah), Academic Press, pp.\,225--237 (1973).

\bibitem{W82} S.\,J. Wolfe. On a continuous analogue of the stochastic 
difference equation
$X_n=\rho X_{n-1} +B_n$, {\it Stoch. Proc. Appl.} {\bf 12}, 301--312 (1982).
\end{thebibliography}
\end{document}